\documentclass[oneside]{amsart}
\usepackage{enumerate,amssymb}
\newtheorem{thm}{Theorem}[section]
\newtheorem{coro}[thm]{Corollary}
\newtheorem{prop}[thm]{Proposition}

\newtheorem{lem}[thm]{Lemma}
\newtheorem*{thmn}{Theorem}
\theoremstyle{definition}
\newtheorem{defn}[thm]{Definition}

\newtheorem{rem}[thm]{Remark}
\newtheorem{question}[thm]{Question}
\newcommand{\Rset}{\mathbb{R}}

\newcommand{\Nset}{\omega}
\newcommand{\Zset}{\mathbb{Z}}

\newcommand{\Cset}{2^\Nset}

\newcommand{\Pset}{\Nset^\Nset}
\newcommand{\UPset}{\Nset^{\uparrow\Nset}}

\newcommand{\eps}{\varepsilon}
\newcommand{\del}{\delta}
\newcommand{\subs}{\subseteq}
\newcommand{\sups}{\supseteq}
\newcommand{\clos}[1]{\overline{#1}}
\newcommand{\dist}{\underline{d}}

\newcommand{\rest}{{\restriction}}
\renewcommand{\leq}{\leqslant}
\renewcommand{\geq}{\geqslant}

\newcommand{\fmany}{\forall^\infty}
\newcommand{\emany}{\exists^\infty}
\newcommand{\upto}{{\nearrow}}
\newcommand{\grG}{\mathbb{G}}

\newcommand{\abs}[1]{\lvert#1\rvert}
\newcommand{\norm}[1]{\lVert#1\rVert}
\newcommand{\seq}[1]{\langle#1\rangle}
\newcommand{\seqeps}{\seq{\eps_n}}
\newcommand{\Seqeps}{\seq{\eps_n}\in(0,\infty)^\omega}
\newcommand{\si}{$\sigma$\nobreakdash-}
\newcommand{\hm}{\mathcal H}
\newcommand{\uhm}{\overline{\mathcal H}}

\DeclareMathOperator{\non}{\mathsf{non}}
\DeclareMathOperator{\add}{\mathsf{add}}
\DeclareMathOperator{\cov}{\mathsf{cov}}

\DeclareMathOperator{\diam}{diam}
\DeclareMathOperator{\proj}{proj}

\newcommand{\smz}{\ensuremath{\boldsymbol{\mathsf{Smz}}}}
\newcommand{\ssmz}{\ensuremath{{\boldsymbol{\mathsf{Smz}}^\sharp}}}

\newcommand{\MM}{\mathcal M}

\newcommand{\EE}{\mathcal E}

\newcommand{\JJ}{\mathcal J}
\newcommand{\bbb}{\mathfrak b}

\newcommand{\nonM}{\non\MM}
\newcommand{\addM}{\add\MM}

\newcommand{\mc}[1]{\mathcal{#1}}

\newenvironment{enum}{\begin{enumerate}[\rm(i)]}{\end{enumerate}}
\newenvironment{itemyze}%
{\begin{list}{\textbullet}{\setlength{\labelwidth}{1ex}\setlength{\leftmargin}{2.1em}}}%
{\end{list}}

\newcommand{\Implies}{\ensuremath{\Rightarrow}}
\makeatletter
\def\vardim@#1#2{\mathop{\vtop{\ialign{##\crcr
 \hfil$#1\m@th\operator@font dim$\hfil\crcr
 \noalign{\nointerlineskip\kern\ex@}#2#1\crcr
 \noalign{\nointerlineskip\kern-\ex@}\crcr}}}}
\def\varinjdim{\mathpalette\vardim@\rightarrowfill@}
\makeatother

\begin{document}
\title
[Meager-additive sets in topological groups]
{Meager-additive sets in topological groups}
\author{Ond\v rej Zindulka}
\address
{Department of Mathematics\\
Faculty of Civil Engineering\\
Czech Technical University\\
Th\'akurova 7\\
160 00 Prague 6\\
Czech Republic}
\email{zindulka@mat.fsv.cvut.cz}
\urladdr{http://mat.fsv.cvut.cz/zindulka}

\thanks{Work on this project was partially conducted during the author's sabbatical stay
at the Instituto de matem\'aticas, Unidad Morelia,
Universidad Nacional Auton\'oma de M\'exico supported by CONACyT grant no.~125108.
}

\begin{abstract}
By the Galvin-Mycielski-Solovay theorem, a subset $X$ of the line has
Borel's strong measure zero if and only if $M+X\neq\Rset$ for each meager set $M$.

A set $X\subs\Rset$ is \emph{meager-additive} if $M+X$ is meager for each meager set $M$.
Recently a theorem on meager-additive sets that perfectly parallels
the Galvin-Mycielski-Solovay theorem was proven: A set $X\subs\Rset$
is meager-additive if and only if it has sharp measure zero, a notion
akin to strong measure zero.

We investigate the validity of this result in Polish groups.
We prove, e.g., that a set in a locally compact Polish group admitting an invariant
metric is meager-additive if and only if
it has sharp measure zero.
We derive some consequences and calculate some cardinal invariants.
\end{abstract}
\maketitle

\section{Introduction}

100 years ago \'Emile Borel~\cite{MR1504785} defined the notion of
strong measure zero: a metric space $X$ has
\emph{strong measure zero} (hereafter \smz) if for any sequence $\seqeps$
of positive numbers there is a cover $\{U_n\}$ of $X$
such that $\diam U_n\leq\eps_n$ for all $n$.
In the same paper, Borel conjectured that every \smz{} set of reals is countable.
This statement, known as the \emph{Borel Conjecture}, attracted
a lot of attention.
It is well-known that the Borel Conjecture is independent of ZFC, the usual axioms
of set theory.
The proof of consistency of its failure was settled by 1948
by Sierpi\'nski~\cite{SIERPINSKI}, who proved in 1928
that the Luzin set (that exists under the Continuum Hypothesis) is a counterexample,
and G\"odel~\cite{MR0002514},
who proved in 1948 the consistency of the Continuum Hypothesis.
The consistency of the Borel Conjecture was proved in 1976 in Laver's ground-breaking
paper~\cite{MR0422027}.

Over time numerous characterizations of \smz{} were discovered.
Maybe the most interesting is the Galvin--Mycielski--Solovay Theorem:
Confirming Prikry's conjecture, Galvin, Mycielski and Solovay~\cite{GMS}
proved the following:

\begin{thmn}[\cite{GMS}]\label{GMS}
A set $X\subs\Rset$ is \smz{} if and only if $X+M\neq\Cset$
for each meager set $M\subs\Cset$.
\end{thmn}

This theorem led to the study of small sets defined in a similar manner.
We focus on the meager-additive sets. By definition, a set $X\subs\Rset$
is \emph{meager-additive} if $X+M$ is meager for each meager set $M$.
The definition easily extends to any topological group.
Meager-additive sets in $\Cset$ have received a lot of attention.
They were investigated by many, most notably by
Bartoszy\'nski and Judah~\cite{MR1350295}, Pawlikowski~\cite{MR1380640}
and Shelah~\cite{MR1324470}.

In a recent paper~\cite{Zin-prism}, the author of the present paper used
combinatorial properties of meager-additive sets described by
Shelah~\cite{MR1324470} and Pawlikowski~\cite{MR1380640} to characterize
meager-additive sets in $\Cset$ in a way that nicely parallels
Borel's definition of strong measure zero, thus obtaining a theorem
that is very similar to the Galvin--Mycielski--Solovay Theorem.

In more detail, a set $X$ has sharp measure zero
(thereinafter \ssmz) if for any sequence $\seqeps$
of positive numbers there is a \emph{$\gamma$-groupable} cover $\{U_n\}$ of $X$
such that $\diam U_n\leq\eps_n$ for all $n$.
Note that the only difference between strong measure zero and sharp measure zero
is that the cover is required to be $\gamma$-groupable (the notion is defined below).
With these definitions, the theorem reads
\begin{thmn}[{\cite[Theorem 5.1]{Zin-prism}}]
A set $X\subs\Cset$ has sharp measure zero if and only if it is meager-additive.
\end{thmn}
With some effort, the theorem was shown to hold
also in $\Rset$ and Euclidean spaces.
In the present paper we attempt to extend it to a wider class of Polish groups.

The inspiration comes from recent results on strong measure zero:
Kysiak \cite{kysiak} and Fremlin \cite{fremlin5} showed that
the Galvin-Mycielski-Solovay Theorem holds in all locally compact Polish groups.
Hru\v s\'ak, Wohofsky and Zindulka~\cite{MR3453581}
and Hru\v s\'ak and Zapletal~\cite{MR3707641} found, roughly speaking, that under the
Continuum Hypothesis the Galvin--Mycielski--Solovay Theorem fails
for groups that are not locally compact.

This of course raises questions about the scope of the above theorem.
Does it hold for locally compact metric groups?
Does it consistently fail for other Polish groups?
We give partial answers to the former questions; to date, the latter remains a mystery.

It turns out that in addition to sharp measure zero and meager-additive
sets it is handy to consider a subfamily of meager-additive sets, the so called \emph{sharply meager-additive} sets
(see Section \ref{sec:sharply} for the definition).

The following are the main results of the paper.

\begin{thm}\label{thm1}
Let $\grG$ be a Polish group equipped with a right-invariant metric.
Then every sharply $\MM$-additive set in $\grG$ is $\ssmz$.
\end{thm}

\begin{thm}\label{thm2}
Let $\grG$ be a locally compact Polish group and $S\subs\grG$.
Then $S$ is $\ssmz$ if and only if $S$ is sharply $\MM$-additive.
\end{thm}

\begin{thm}\label{thm3}
Let $\grG$ be a Polish group admitting an invariant metric.
If $X\subs\grG$ is $\MM$-additive, then $X$ is $\ssmz$ (in any metric on $\grG$).
\end{thm}

As a corollary to these theorems, we get an equivalence of the three properties in
locally compact groups admitting an invariant metric. This class of groups
includes all compact or abelian Polish groups.

\begin{thm}\label{thm4}
Let $\grG$ be a locally compact group admitting an invariant metric.
Let $X\subs\grG$. The following are equivalent:
\begin{enum}
\item $X$ is $\MM$-additive,
\item $X$ is sharply $\MM$-additive,
\item $X$ is $\ssmz$.
\end{enum}
\end{thm}

Consequences include, for instance:
\begin{itemyze}
  \item Meager-additive sets are preserved by continuous mappings
  between groups (Theorem~\ref{mapping22}).
  \item A product of two meager-additive sets is meager-additive
  (Theorem~\ref{cart}).
  \item Meager-additive sets are universally meager and transitively meager
  (Propositions~\ref{umg} and \ref{trans}).
  \item All $\gamma$-sets are meager-additive (Proposition \ref{gamma}).
\end{itemyze}

We also calculate the uniformity number of the ideal of meager-additive sets
in any Polish group admitting an invariant metric (Theorem \ref{invariants}).

\medskip

The paper is organized as follows. In Section \ref{sec:ssmz} we recall
elementary material on sharp measure zero and a few technical facts.
In Section \ref{sec:sharply} we introduce sharply meager-additive sets
and prove Theorem \ref{thm1}. In Section \ref{sec:loccpt} we show
that in locally compact groups the classes of sharp measure zero and sharply
meager-additive sets coincide, i.e., we prove Theorem \ref{thm2}.
In Section \ref{sec:polish} Theorems \ref{thm3} and \ref{thm4} are proved.
In Section \ref{sec:conseq} we derive the aforementioned consequences and in Section
\ref{sec:uniformity} we calculate the uniformity number of meager-additive sets.
The last section lists some open problems that we consider interesting.

\medskip

Some common notation used throughout the paper includes $\abs{A}$ for
the cardinality of a set $A$, $\Nset$ for the set of natural numbers,
$\Pset$ for the family of all sequences of natural numbers.
We also write $E_n\upto E$ to denote that $\seq{E_n}$ is an increasing sequence
of sets with union $E$.
All metric spaces and topological groups we consider are separable.
The diameter of a set $E$ in a metric space is denoted $\diam E$.
A closed ball of radius $r$ centered at $x$ is denoted by $B(x,r)$.

\section{Sharp measure zero: review}\label{sec:ssmz}

In this section we review a few facts from the theory of sharp measure zero,
as developed in~\cite{Zin-prism}. There are many equivalent definitions
listed there. We will make use of two of them.

Let's recall once again that by Borel's definition~\cite{MR1504785}
a metric space $X$ has
\emph{strong measure zero} (abbreviated as $\smz$) if
for every sequence $\seqeps$ of positive reals there is a cover
$\seq{U_n}$ of $X$ such that $\diam U_n\leq\eps_n$ for all $n$.
Our first goal is to describe sharp measure zero in a way that parallels
this definition.

We need to recall a few notions regarding covers.
Let $\seq{U_n}$ be a sequence of subsets of a set $X$.
\begin{itemyze}
\item
A sequence of sets $\seq{U_n}$ is called
a \emph{$\gamma$-cover} if each $x\in X$ belongs to all but finitely
many $U_n$.
\item
A sequence of sets $\seq{U_n}$ is called a \emph{$\gamma$-groupable cover}
if there is a partition $\Nset=I_0\cup I_1\cup I_2\cup\dots$
into consecutive finite intervals (i.e., $I_{j+1}$ is on the right of $I_j$ for all $j$)
such that the sequence $\seq{\bigcup_{n\in I_j}U_n:j\in\Nset}$ is a $\gamma$-cover.
The partition $\seq{I_j}$ will be occasionally called \emph{witnessing}
and the finite families $\{U_n:n\in I_j\}$
will be occasionally called \emph{witnessing groups}.
\end{itemyze}

\begin{defn}
A metric space $X$ has \emph{sharp measure zero} (abbreviated as $\ssmz$) if
for every sequence $\seqeps$ of positive reals there is a $\gamma$-groupable cover
$\seq{U_n}$ of $X$ such that $\diam U_n\leq\eps_n$ for all $n$.
\end{defn}

For a metric space $X$, the family of all $\ssmz$ subsets of $X$ will be denoted
$\ssmz(X)$.

Sharp measure zero is a \si additive property and it is preserved
by uniformly continuous maps:
\begin{prop}[{\cite[1.2]{Zin-prism}}]\label{trivUH}
\begin{enum}
\item
For a metric space $X$, the family $\ssmz(X)$ is a \si ideal.
\item
If $f:X\to Y$ is a uniformly continuous mapping and $S\subs X$ is $\ssmz$, then
$f(S)$ is $\ssmz$.
\end{enum}
\end{prop}
The latter yields another property of $\ssmz$ that is worth mentioning:
$\ssmz$ is a \emph{uniform} property: it depends only on the uniformity
induced by a metric, not the metric itself.
Also, it is not a topological property; we prove that in Remark~\ref{nottop}.

\begin{lem}[{\cite[3.4)]{Zin-prism}}]\label{totbd}
\begin{enum}
\item Every $\ssmz$ set admits a countable cover by totally bounded sets.
\item In particular, every \ssmz{} set $S$ in a complete metric space
is contained in a \si compact set.
\end{enum}
\end{lem}
This lets us relax the uniform continuity condition in Proposition~\ref{trivUH}(ii):

\begin{prop}\label{mapping1}
Let $X$ be a complete or \si compact metric space and
$f:X\to Y$ a continuous mapping. If $S\subs X$ is $\ssmz$, then so is
$f(S)$.
\end{prop}
\begin{proof}
Lemma~\ref{totbd} yields, in either case, a countable cover $\{K_n\}$ of $S$ by
compact sets.
The maps $f\rest K_n$ are thus uniformly continuous for all $n\in\Nset$.
Hence $f(S\cap K_n)$ are $\ssmz$ by ~\ref{trivUH}(ii) and $f(S)$
is $\ssmz$ by~\ref{trivUH}(i).
\end{proof}

\subsection*{Hausdorff measures}

The other definition of $\ssmz$ that we will utilize
is set up in terms of Hausdorff measure and its modification.
We recall the definitions and basic facts.

A non-decreasing, right-continuous function $h:[0,\infty)\to[0,\infty)$
such that $h(0)=0$ and $h(r)>0$ if $r>0$ is called a \emph{gauge}.

If $\del>0$, a cover $\mc A$ of a set $E\subs X$ is termed
\emph{$\del$-fine} if $\diam A\leq\del$ for all $A\in\mc A$.
Let $h$ be a gauge.
For each $\del>0$ set
$$
  \hm^h_\delta(E)=
  \inf\left\{\sum_{n\in\Nset}h(\diam E_n):
  \text{$\{E_n\}$ is a countable $\delta$-fine cover of $E$}\right\}
$$
and define
$$
  \hm^h(E)=\sup_{\delta>0}\hm^h_\delta(E).
$$
The set function $\hm^h$ is called the \emph{$h$-dimensional Hausdorff measure}.
Properties of Hausdorff measures are well-known. The following, including
the two propositions, can be found e.g.~in~\cite{MR0281862}.
The restriction of $\hm^h$ to Borel sets is a $G_\del$-regular Borel measure.

\medskip

Sharp measure zero requires a slight modification of Hausdorff measure,
as it was introduced in~\cite{Zin-prism}.
Let $h$ be a gauge. For each $\delta>0$ set
$$
  \uhm^h_\delta(E)=
  \inf\left\{\sum_{n=0}^N h(\diam E_n):
  \text{$\{E_n\}$ is a \emph{finite} $\delta$-fine cover of $E$}\right\}.
$$
Note that the only difference is that only finite covers are taken in account,
as opposed to countable covers in the definition of $\hm^h_\delta(E)$.
Then let
$$
  \uhm_0^h(E)=\sup_{\delta>0}\uhm^h_\del(E).
$$
It is easy to check that $\uhm_0^h$ is a finitely subadditive set function.
However, it is not a measure, since it need not be \si additive.
That is why another step is required: we need to apply the operation
known as Munroe's \emph{Method I construction}
(cf.~\cite{MR0053186} or \cite{MR0281862}):
$$
  \uhm^h(E)=\inf\Bigl\{\sum_{n\in\Nset}\uhm^h_0(E_n):
  E\subs\bigcup_{n\in\Nset}E_n\Bigr\}.
$$
Thus defined set function is indeed an outer measure whose restriction to
Borel sets is a Borel measure.
It is called the \emph{$h$-dimensional upper Hausdorff measure}.

Detailed information on the upper Hausdorff measure is provided in~\cite{Zin-prism},
see also~\cite{MR2957686}.
Here we only recall two properties that we will make use of.
\begin{lem}[{\cite[Lemma 3.4]{Zin-prism}}]\label{lem1}
Let $h$ be a gauge and $E$ a set in a metric space.
If $E$ is \si compact, then $\uhm^h(E)=\hm^h(E)$.
\end{lem}
\begin{lem}[{\cite[Lemma 3.9]{Zin-prism}}]\label{gammagr}
Let $h$ be a gauge and $E$ a set in a metric space. If
$E$ has a $\gamma$-groupable cover $\seq{U_n}$ such that
$\sum_{n\in\Nset}h(\diam U_n)<\infty$, then $\uhm^h(E)=0$.
\end{lem}

By a classical result of Besicovitch~\cite{MR1555386,MR1555389} (cf.~\cite{Zin-prism}
for detailed proofs), a metric space $X$ is $\smz$ if and only if $\hm^h(X)=0$
for all gauges $h$. It is nice that $\ssmz$ sets are characterized in exactly
the same way, only that the upper Hausdorff measures are used
in place of Hausdorff measures:
\begin{thm}[{\cite[Theorem 3.7]{Zin-prism}}]\label{basicUhnull}
A metric space $X$ is $\ssmz$ if and only if
$\uhm^h(X)=0$ for each gauge $h$.
\end{thm}

\section{Sharply $\MM$-additive sets}\label{sec:sharply}

Given a metric space $X$, we write $\MM(X)$ to denote the
family of all meager sets in $X$, and likewise for a topologigal group.
In most cases there is no danger of confusion, so we write only
$\MM$ instead.

Let us clarify our terminology regarding Polish groups first.
A \emph{Polish group} is a topological group that is homeomorphic to a
Polish space, i.e., a completely metrizable separable topological space.
Trivially, every Polish group admits a complete metric. It also admits
a right-invariant metric. The two metrics, though, do not have to coincide:
there are Polish groups that do not have a metric that is simultaneously
right-invariant and complete.
Also, not every Polish group has a (two-sided) invariant metric.
However, if the group is abelian or compact, then it has an invariant metric.
Any invariant metric on a Polish group is complete.

Recall that Prikry~\cite{prikry} proved that if $\grG$ is a separable group with a
right-invariant metric and $X\subs\grG$ is a set satisfying
$\forall M\in\MM\ MX\neq\grG$, then $X$ has strong measure zero.
This is the easy part of the Galvin-Mycielski-Solovay Theorem
discussed above.
In this section we prove a counterpart of Prikry's result for
sharp measure zero. We first provide a formal definition of
an $\MM$-additive set.

\begin{defn}
Let $\grG$ be a Polish group and $X\subs\grG$. The set $X$ is called
\emph{$\MM$-additive} (or \emph{meager-additive}) if for every meager
set $M\subs\grG$, the set $MX$ is meager.

We will write $\MM^*(\grG)$ for the \si ideal of $\MM$-additive sets in $\grG$.
\end{defn}
$\MM$-additive sets are difficult to deal with in a general
Polish group. The following property, though appearing more complex,
is much easier, as we shall see.

\begin{prop}
Let $\grG$ be a Polish group and $X\subs\grG$.
The following are equivalent.
\begin{enum}
\item $\forall M\in\MM\ \exists K\sups X \text{ \si compact } MK\neq\grG$,
\item $\forall M\in\MM\ \exists K\sups X \text{ \si compact } MK\in\MM$.
\end{enum}
\end{prop}
\begin{proof}
Let $X$ satisfy (ii) and $M\in\MM$. We may suppose $M$ is $F_\sigma$.
Let $Q\subs\grG$ be countable and dense.
Then $QM$ is meager, hence there is a \si compact set $K\sups X$
and $z\notin QMK$. The latter yields $Q^{-1}z\cap MK=\emptyset$.
Since $Q$ is dense, so is $Q^{-1}z$, thus $MK$ is disjoint with a dense set.
If we prove that $MK$ is $F_\sigma$, we are done. To that end it is enough show that
if $C$ is compact and $F$ is closed, then $FC$ is closed.
So suppose $y\notin FC$. Let $d$ be a left-invariant metric on $\grG$.
For each $x\in C$ the set $Fx$ is closed and thus there is $\eps_x>0$ such that
$d(Fx,y)>\eps_x$.
Since $C$ is compact, there is a finite set $I\subs C$ such that
the family of balls $\{B(x,\frac{\eps_x}2):x\in I\}$ covers $C$.
Let $\eps=\min_{x\in I}\frac{\eps_x}2$. We claim that $d(FC,y)\geq\eps$.
Otherwise there is $x\in I$ such that
$d(F\cdot B(x,\frac{\eps_x}2),y)<\frac{\eps_x}2$.
Since $d$ is left-invariant, we have
$F\cdot B(x,\frac{\eps_x}2)\subs B(Fx,\frac{\eps_x}2)$. Therefore
$d(B(Fx,\frac{\eps_x}2),y)<\frac{\eps_x}2$. Hence there is $z$ such that
$d(Fx,z)\leq\frac{\eps_x}2$ and $d(z,y)<\frac{\eps_x}2$, which in turn
yields $d(Fx,y)<\eps_x$, a contradiction.
We proved that $d(FC,y)>0$ for each $y\notin FC$, i.e., $FC$ is closed,
which concludes the proof.
\end{proof}
It is clear that the property described by this proposition is stronger than
$\MM$-additivity. We capture it in the following definition.

\begin{defn}
Let $\grG$ be a Polish group and $X\subs\grG$. The set $X$ is called
\emph{sharply $\MM$-additive} if for every meager set $M\subs\grG$ there is
a \si compact set $K\sups X$ such that $MK$ is meager.

We will write $\MM^\sharp(\grG)$ for the \si ideal of sharply $\MM$-additive
sets in $\grG$.
\end{defn}
As proved in~\cite{Zin-prism}, in $\Cset$ and $\Rset$ and their finite powers,
every $\MM$-additive set is sharply $\MM$-additive. We shall see later that
the same remains true in a wide class of locally compact Polish groups.

Our first theorem shows that sharply $\MM$-additive sets are of sharp
measure zero.

\begin{thm}\label{main1}
Let $\grG$ be a Polish group equipped with a right-invariant metric.
Then every sharply $\MM$-additive set in $\grG$ is $\ssmz$.
\end{thm}
\begin{proof}
Let $d$ be a right-invariant metric on $\grG$
and let $X\subs\grG$ be a sharply $\MM$-additive set.
Let $g$ be an arbitrary gauge. We want to show that $\uhm^g(X)=0$.
Let $Q\subs\grG$ a countable dense symmetric set.
Since Hausdorff measures are $G_\delta$- regular,
there is a $G_\delta$-set $H\sups D$ such that $\hm^g(H)=0$.
Since $D$ is symmetric, we may, replacing $H$ with $H\cap H^{-1}$ if necessary,
suppose that $H$ is symmetric as well. The set $M=\grG\setminus H$
is obviously meager.
Since $X$ is sharply $\MM$-additive, there is a \si compact set $K\sups X$
and $z\notin MK$. Routine manipulation leads to $K\subs H^{-1}z=Hz$.
Since the underlying metric is right-invariant, with the aid
of Lemma~\ref{lem1} we get
$$
  \uhm^g(X)\leq\uhm^g(K)=\hm^g(K)\leq\hm^g(Hz)=\hm^g(H)=0,
$$
as required.
\end{proof}

\section{Sharply $\MM$-additive sets in locally compact groups}
\label{sec:loccpt}

In this section we prove that in locally compact Polish groups,
sharp meager-additivity and sharp measure zero
are equivalent notions. The following theorem may be regarded a
``Galvin-Mycielski-Solovay theorem for $\MM$-additive sets''.

\begin{thm}\label{thm:action3}
Let $\grG$ be a locally compact Polish group and $S\subs\grG$.
Then $S$ is $\ssmz$ if and only if $S$ is sharply $\MM$-additive.
\end{thm}
It turns out that this theorem is a particular case of a more general result.
Its proof is postponed until the end of the section.
Let us first note that it follows from Proposition~\ref{mapping1} that
the notion of sharp measure zero in a locally compact
metric group is independent of the choice of metric.
\begin{prop}\label{indep1}
Let $d_1$ and $d_2$ be two metrics on a locally compact metrizable group $\grG$.
If $X\subs\grG$ is $\ssmz$ with respect to $d_1$,
then so it is with respect to $d_2$.
\end{prop}
Thus we may refer to a $\ssmz$ set in a locally compact group without referring
to a metric on $\grG$.

The following theorem is a group-free version of Theorem \ref{thm:action3}.

\begin{thm}\label{thm:action}
Let $Y$ be a Polish metric space and $X$ a separable locally compact metric space.
Let $\phi:Y\times X\to X$ be a continuous mapping such that
for each $y\in Y$ and every compact nowhere dense set $P\subs X$ the image
$\phi(\{y\}\times P)$ is nowhere dense.
If $S\subs Y$ is $\ssmz$ and $M\subs X$ is meager,
then there is a \si compact set $\widehat S\supseteq S$ such that
$\phi(\widehat S\times M)$ is meager.
\end{thm}
\begin{proof}
We need the following combinatorial lemma.
\begin{lem}[{\cite[Lemma 6.2]{MR3707641}}]\label{comp}
Let $K$ be a compact metric space and
$X$ a separable locally compact metric space.
Let $U\subs X$ be an open set with compact closure $C=\clos U$ and
$P\subs X$ be compact nowhere dense.
Let $\phi:K\times X\to X$ be a continuous mapping such that
for each $y\in K$ the image $\phi(\{y\}\times P)$ is nowhere dense.
Then
$$
  \forall\eps>0\ \exists\del>0\ \forall x\in C\ \forall y\in K\ \exists z\in C
  \quad B(z,\del)\subs B(x,\eps)\setminus\phi(B(y,\del)\times P).
$$
\end{lem}

Since $S$ is $\ssmz$ and $Y$ is complete, $S$ is contained in
a \si compact subset of $Y$, so we may suppose $Y$ is \si compact.
Let $K_n$ be compact sets in $Y$ with $K_n\upto Y$ and
let $P_n$ be compact nowhere dense sets with $P_n\upto M$.
Let $\{U_k\}$ be a countable base of $X$.
For each $k$ choose $x_0^k\in X$ and $\eps_0^k>0$ such that
$B(x_0^k,\eps_0^k)\subs U_k$ is compact, and let $C_k=B(x_0^k,\eps_0^k)$.
Use Lemma~\ref{comp} to recursively construct a sequence
$\Seqeps$ such that
\begin{equation}\label{eq1}
\begin{split}
\forall n\ \forall i\leq n\
  & \forall x\in C_i \forall y\in K_n\ \exists z\in C_i\\
   & B(z,\eps_n)\subs B(x,\eps_{n-1})\setminus\phi((B(y,\eps_n)\cap K_n)\times P_n).
\end{split}
\end{equation}
Since $S$ is $\ssmz$, there is an  $\gamma$-groupable cover $\{E_n\}$ of $S$
such that $\diam E_n<\eps_n$ for all $n$.
Hence for each $n$ there is $y$ such that $E_n\subs B(y,\eps_n)$.
Therefore we may use \eqref{eq1} to construct sequences $\seq{x_n^k:n\in\Nset}$
such that for all $k$
\begin{equation}\label{eq2}
  B(x_{n+1}^k,\eps_{n+1})\subs B(x_n^k,\eps_n)
  \setminus \phi((E_{n+1}\cap K_{n+1})\times P_{n+1}).
\end{equation}
It is easy to check that since $\{E_n\}$ is a $\gamma$-groupable cover of $S$
and $K_n\upto Y$, the family $\{E_n\cap K_n\}$ is also a $\gamma$-groupable cover of $S$.
Thus we might have supposed that $E_n\subs K_n$, and also that all $E_n$'s are closed.
Therefore \eqref{eq2} simplifies to
\begin{equation}\label{eq3}
  B(x_{n+1}^k,\eps_{n+1})\subs B(x_n^k,\eps_n)
  \setminus \phi(E_{n+1}\times P_{n+1}).
\end{equation}
In particular, $B(x_n^k,\eps_n)$ is a decreasing sequence of compact balls for all $k$
and thus there is a point $x^k\in U_k$ such that
\begin{equation}\label{eq4}
x^k\notin \bigcup_{n\in\Nset}\phi(E_n\times P_n).
\end{equation}
The set $\widehat S$ is constructed as follows: Let $\mc G_j$ be the groups of $E_n$'s
witnessing to the $\gamma$-groupability of $\{E_n\}$.
Put $G_n=\bigcap_{n\in\mc G_j}E_n$ and let
$F_n=\bigcap_{i<n}G_i$ and $\widehat S=\bigcup_{n\in\Nset}F_n$.
It is clear that since $E_n$'s are closed, the set $\widehat S$ is $F_\sigma$, and
clearly $S\subs\widehat S$.
Moreover, routine calculation shows that
$\widehat S\times M\subs\bigcup_n E_n\times P_n$. Therefore \eqref{eq4} yields
$x^k\notin\phi(\widehat S\times M)$ for all $k$.
So letting $D=\{x^k:k\in\Nset\}$, the set $D$ is disjoint with
$\phi(\widehat S\times M)$ and it is dense in $X$.
Since $\widehat S$ and $M$ are \si compact, so is $\phi(\widehat S\times M)$.
Therefore $\phi(\widehat S\times M)$ is an $F_\sigma$ set disjoint with a dense set,
and is thus meager.
\end{proof}

The following is obviously a particular case of the theorem.
\begin{coro}\label{thm:action2}
Let $\phi:\grG\curvearrowright X$ be an action of a Polish group
$\grG$ on a \si compact space $X$. If $S\subs\grG$ is $\ssmz$, then there is
\si compact set $\widehat S\supseteq S$ such that
$\phi(\widehat S\times M)$ is meager.
\end{coro}

\begin{proof}[Proof of Theorem~\ref{thm:action3}]
Let $X=\grG$, define $\phi:\grG\curvearrowright\grG$ by
$\phi(x,y)=yx$ and apply the corollary to get the forward implication.
The reverse implication follows from Theorem~\ref{main1}.
\end{proof}

\section{Meager-additive sets}\label{sec:polish}

In the previous sections we linked sharp measure zero to sharp meager-additivity.
In this section we investigate $\MM$-additive sets in the hope that we can link them to
$\ssmz$ sets.

We first note that if a set in a Polish group is \ssmz{} in some complete metric,
then it is \ssmz{} in any other metric as well. This follows at once from
Proposition~\ref{mapping1}.
Thus in the following theorem the particular metric does not matter.

\begin{thm}\label{polish1}
Let $\grG$ be a Polish group admitting an invariant metric.
If $X\subs\grG$ is $\MM$-additive, then $X$ is $\ssmz$ (in any metric on $\grG$).
\end{thm}

Let us recall once again that not every Polish group admits an invariant metric,
but if it is compact or abelian, then it does.
Also, any invariant metric on a Polish group is complete.
Polish groups that admit an invariant metric are referred to as \emph{tsi}
groups.

The following easily follows from Theorems~\ref{polish1} and~\ref{thm:action3}.
\begin{thm}\label{main3a}
Let $\grG$ be a locally compact \emph{tsi} group.
Let $X\subs\grG$. The following are equivalent:
\begin{enum}
\item $X$ is $\MM$-additive,
\item $X$ is sharply $\MM$-additive,
\item $X$ is $\ssmz$.
\end{enum}
\end{thm}

The rest of this section is devoted to the proof of Theorem~\ref{polish1}.

The symbol $\grG$ stands for the Polish group. We suppose it is equipped with
an invariant (and thus complete) metric throughout this section.
The metric will be denoted by $d$.
We will use additive notation for the group operation, though $\grG$ is not a priori
assumed to be abelian.

We shall deal with series in $\grG$. Write $\norm{x}=d(0,x)$. Suppose that $\seq{x_n}$ is
a sequence in $\grG$ such that $\sum_{n\in\omega}\norm{x_n}<\infty$. Then the sequence
of partial sums is Cauchy and therefore converges. We will thus write $\sum_{n\in\omega}x_n$
for $\lim_{n\to\infty}(x_0+x_1+\dots+x_n)$.

Given $\eps>0$, a set $B\subs\grG$ is $\eps$-separated if $d(p,q)\geq\eps$ for any
two distinct points in $B$.

\subsection*{The grid}
For each $n\in\omega$ set $\eps_n=2^{-n}$. We use this notation throughout.
Define a system of grids in $\grG$ as follows.
\begin{itemyze}
  \item Let $Q_0$ be a set that is maximal among all $\eps_0$-separated sets
  in $\grG$ containing $0$
  \item For $n\geq1$ let $Q_n$ be a set that is maximal among all $\eps_n$-separated
  subsets of the closed ball $B(0,\eps_{n-1})$ containing $0$.
\end{itemyze}
The maximality ensures that any point in $B(0,\eps_{n-1})$ is within distance $\eps_n$ from some
point of $Q_n$.

Each $Q_n$ is finite or countable. Provide it with a well-ordering
so that $0$ is the first point in the order and
\begin{itemyze}
  \item let $Q_{n,m}$ be the set of the first $m$ points in $Q_n$.
\end{itemyze}
Then set
\begin{itemyze}
  \item $Q_n^*=Q_{0,n}+Q_{1,n}+\dots+Q_{n,n}$,
  \item $Q=\bigcup_{n\in\omega}Q_n^*$.
\end{itemyze}
It is easy to check that $Q$ is a dense set.

\subsection*{Base for meager sets}

Next we define canonical meager sets in $\grG$.
We will use the following notation: $\Pset$ is the family of all sequences
of natural numbers; $\UPset\subs\Pset$ is the family of nondecreasing
unbounded sequences of natural numbers.
The quantifiers $\fmany$ and $\emany$ have the usual meaning:
``for all except finitely many'' and ``for infinitely many''.
The symbol $\dist$ denotes the lower distance.

\begin{defn}
Let $f\in\UPset$, $x\in\grG$ and $c>0$. Let
$$
  H(f,x,c)=\{z\in\grG:\fmany n\ \dist(z,Q^*_{f(n)}+x)\geq c\,\eps_{f(n+1)}\}.
$$
\end{defn}
\begin{lem}\label{polish3}
The set $H(f,x,c)$ is meager for any $f\in\UPset$, $x\in\grG$ and $c>0$.
\end{lem}
\begin{proof}
$H(f,x,c)$ is disjoint with the set
$$
  \{z\in\grG:\emany n\ z\in Q^*_{f(n)}+x\}=\bigcap_{n\in\omega}\bigcup_{m\geq n}(Q^*_{f(m)}+x)=Q+x.
$$
Since $Q$ is dense in $\grG$, so is $Q+x$. On the other hand,
$H(f,x,c)$ is obviously $F_\sigma$. Hence it is meager.
\end{proof}
Not only that the sets $H(f,x,c)$ are meager, they actually form a base
for meager sets in $\grG$. We need yet a little more:
\begin{lem}\label{polish4}
$\forall M\in\MM\ \forall f\in\UPset\ \exists g\in\UPset\ \exists x\in\grG\
M\subs H(f{\circ} g,x,1)$
\end{lem}
\begin{proof}
Suppose $F_n$ are closed nowhere dense sets such that $F_n\upto M$.
Write $G_n=\grG\setminus F_n$.

We recursively define a sequence $\seq{x_n}$ in $\grG$ and an increasing
sequence $\seq{k_n}$ of integers subject to the following conditions.
Write $y_n=x_0+x_1+\dots+x_n$.
\begin{align}
  &x_{n+1}\in Q_{f(k_n)+1}+Q_{f(k_n)+2}+\dots+Q_{f(k_{n+1})},\label{x1}\\
  &\dist(F_n,Q^*_{f(k_n)}+y_{n+1})>3\eps_{f(k_{n+1})}. \label{x2}
\end{align}
Let $k_0=0$ and $x_0=0$ and proceed by induction.
Suppose that $x_n$ and $k_n$ are defined.
The sets $-x_n-q-G_n$ are open dense for each $q\in\grG$. Since $Q^*_{f(k_n)}$ is finite, there is
$z\in\grG$ and $\del>0$ such that
\begin{equation}\label{dense}
  B(z,\del)\subs\bigcap_{q\in Q^*_{f(k_n)}}(-y_n-q+G_n)\cap B(0,\eps_{f(k_n)}).
\end{equation}
Choose $k_{n+1}>k_n$ large enough to satisfy $4\eps_{f(k_{n+1})}<\del$.

Denote for the moment $m=f(k_n)$ and $j=f(k_{n+1})-m$.
\begin{itemyze}
\item Since $\norm{z}\leq\eps_m$, there is $t_1\in Q_{m+1}$ such that
$d(t_1,z)\leq\eps_{m+1}$.
\item Since $\norm{-t_1+z}\leq\eps_{m+1}$, there is $t_2\in Q_{m+2}$ such that
$d(t_2,-t_1+z)\leq\eps_{m+2}$.\\
$\vdots$
\item Proceed by induction up to $m+j$ to get, for each $i\in[1,j]$, $t_i\in Q_{m+i}$ such that
\begin{equation}\label{eq:t}
\norm{-t_j-t_{j-1}-\dots-t_1+z}\leq\eps_{m+j}.
\end{equation}
\end{itemyze}
Put $x_{n+1}=t_1+t_2+\dots+t_j$. Condition \eqref{x1} is clearly satisfied.

Using left invariance of the metric, inequality~\eqref{eq:t} reads
$d(z,x_{n+1})\leq\eps_{f(k_{n+1})}$.
Therefore $4\eps_{f(k_{n+1})}<\del$ and \eqref{dense} yield
$$
  B(x_{n+1},3\eps_{f(k_{n+1})})\subs B(z,4\eps_{f(k_{n+1})})
  \subs B(z,\del)\subs-y_n-q+G_n
$$
for all $q\in Q^*_{f(k_n)}$.
Using left invariance of the metric again, we get
$B(q+y_{n+1},3\eps_{f(k_{n+1})})\cap F_n=\emptyset$
and \eqref{x2} follows.
Since $Q_{j+1}\subs B(0,\eps_j)$ for all $j$, \eqref{x1} yields
$\norm{x_{n+1}}\leq \sum_{j=f(k_n)}^\infty\eps_j=2\eps_{f(k_n)}$.
In particular,
$\sum\norm{x_n}<\infty$, whence the series $\sum x_n$ converges.
Set $x=\sum x_n=\lim y_n$ and $g(n)=k_n$.

\eqref{x1} also yields $d(x,y_{n+1})\leq 2\eps_{f(k_{n+1})}$. Combine this inequality
with \eqref{x2} and triangle inequality to get, for each $q\in Q^*_{f(k_n)}$,
\begin{align*}
  \dist(F_n,q+x)&\geq\dist(F_n,q+y_{n+1})-d(q+y_{n+1},q+x)\\
                &\geq 3\eps_{f(k_{n+1})}-d(y_{n+1},x)\eps_{f(k_{n+1})}=\eps_{fg(n+1)}.
\end{align*}
It follows that $\dist(F_n,Q^*_{f(k_n)}+x)\geq\eps_{fg(n+1)}$ for all $n$, which is enough.
\end{proof}

\subsection*{Key lemma}
We now prove an important combinatorial lemma on the canonical meager sets.

\begin{lem}\label{polish5}
Let $x,y\in\grG$ and $f,g\in\UPset$. Suppose that for each $k\in\omega$
\begin{align}
  & \abs{Q_{f(k)}}>1,  \label{b1} \\
  & \text{$Q^*_{f(k)}$ is $3\eps_{f(k+1)}$-separated},  \label{b2} \\
  & f(k+1)\geq f(k)+4.  \label{b3}
\end{align}
If $H(f,x,\frac14)\subs H(f{\circ} g,y,1)$, then
$$
  \fmany n\in\omega\ \exists k\in[g(n),g(n+1))\quad \dist(Q^*_{f(k)}+x,Q^*_{f(k)}+y)\leq\eps_{f(k+1)}.
$$
\end{lem}
\begin{proof}
Using right invariance of the metric we may assume $y=0$.
Set
$$
  S=\{n\in\omega:\forall k\in[g(n),g(n+1))\quad \dist(Q^*_{f(k)}+x,Q^*_{f(k)})>\eps_{f(k+1)}\}.
$$
We need to prove that $S$ is finite. So aiming towards contradiction suppose that
$\abs{S}=\Nset$.
Our goal is to construct $\tau\in\grG$ such that
\begin{align}
  &\forall n\in S\quad \dist(\tau,Q^*_{fg(n)})<\eps_{fg(n+1)},  \label{s1}\\
  &\fmany k\in\omega\quad \dist(\tau,Q^*_{f(k)}+x)\geq\frac14\eps_{f(k+1)}. \label{s2}
\end{align}
Once we have such a $\tau$,
since $\abs{S}=\omega$, condition \eqref{s1} will ensure that
$\tau\notin H(f{\circ} g,0,1)$.
On the other hand, condition \eqref{s2} will ensure that
$\tau\in H(f,x,\frac14)$: the desired contradiction.

$\tau$ will be defined as a sum of a series. For each $k\in\omega$ we shall choose
$\tau_k\in Q_{f(k+1)}$ according to the following rules, and then we let
$\tau=\sum_{k\in\omega}\tau_k$.

For any $k$ put $T_k=\sum_{i<k}\tau_i$ and $t_k=\sum_{i\geq k}\tau_i$.

$\tau_k$'s are set up recursively according to the following rules.
\begin{enumerate}[(a)]
\item If there is $n\in S$ such that $k\in[g(n),g(n+1))$, put $\tau_k=0$.
\item If there is $n\notin S$ such that $k\in[g(n),g(n+1))$ and
$\dist(T_k,Q^*_{f(k)}+x)\geq\frac12\eps_{f(k+1)}$, put $\tau_k=0$.
\item If there is $n\notin S$ such that $k\in[g(n),g(n+1))$ and
$\dist(T_k,Q^*_{f(k)}+x)<\frac12\eps_{f(k+1)}$, let $\tau_k$ be the first nonzero element
of $Q_{f(k+1)}$. Condition \eqref{b1} ensures that such a choice is possible.
\end{enumerate}
The rules ensure that
\begin{equation}\label{s3}
    T_k\in Q_{f(1),2}+Q_{f(2),2}+\dots+Q_{f(k),2}\subs Q^*_{f(k)}.
\end{equation}
Notice also that since $\tau_i\in Q_{f(i+1)}$, we have $\norm{\tau_i}\leq2\eps_{f(i+1)}$
and thus condition \eqref{b3} yields
\begin{equation}\label{s4}
    \norm{t_k}\leq\frac{2}{15}\eps_{f(k)}.
\end{equation}

We prove \eqref{s1} first. Suppose $n\in S$.
Since $T_{g(n)}\in Q^*_{fg(n)}$ by \eqref{s3}, we have
$$
  \dist(\tau,Q^*_{fg(n)})\leq d(T_{g(n)}+t_{g(n)},T_{g(n)})=\norm{t_{g(n)}}.
$$
Since $n\in S$, rule (a) yields $t_{g(n)}=t_{g(n+1)}$ and thus \eqref{s4}
gives
$$
  \dist(\tau,Q^*_{fg(n)})\leq\norm{t_{g(n+1)}}\leq \frac{2}{15}\eps_{fg(n+1)}
  <\eps_{fg(n+1)}.
$$
Thus \eqref{s1} is proved.
The proof of \eqref{s2} is split into three cases corresponding to the three rules (a)--(c).
\subsection*{Claim 1}
If $n\in S$ and $k\in[g(n),g(n+1))$,
then $\dist(\tau,Q^*_{f(k)}+x)\geq\frac14\eps_{f(k+1)}$.
\begin{proof}
Since $n\in S$, \eqref{s3} implies
$
  \dist(Q^*_{f(k)}+x,T_k)>\eps_{f(k+1)}.
$
Rule (a) yields $t_k=t_{g(n+1)}$.
Hence
$
  \norm{t_k}=\norm{t_{g(n+1)}}\leq\frac{2}{15}\eps_{fg(n+1)}\leq\frac{2}{15}\eps_{f(k+1)}.
$
Combine and use triangle inequality and left invariance of the metric to get
\begin{align*}
  \dist(Q^*_{f(k)}+x,\tau)
  &\geq\dist(Q^*_{f(k)}+x,T_k)-\dist(T_k,\tau) \\
  &>\eps_{f(k+1)}-\norm{t_k}\geq\eps_{f(k+1)}-\frac{2}{15}\eps_{f(k+1)}
  >\frac14\eps_{f(k+1)}. \qedhere
\end{align*}
\end{proof}

\subsection*{Claim 2}
If $n\notin S$, $k\in[g(n),g(n+1))$ and $\dist(T_k,Q^*_{f(k)}+x)\geq\frac12\eps_{f(k+1)}$,
then $\dist(\tau,Q^*_{f(k)}+x)\geq\frac14\eps_{f(k+1)}$
\begin{proof}
This time rule (b) applies. Since $\tau_k=0$, we have $t_k=t_{k+1}$.
Use the same arguments as before and the assumption.

\begin{align*}
  \dist(Q^*_{f(k)}+x,\tau)
  &\geq\dist(Q^*_{f(k)}+x,T_k)-\dist(T_k,\tau) \\
  &\geq\frac12\eps_{f(k+1)}-\norm{t_k}=\frac12\eps_{f(k+1)}-\norm{t_{k+1}}\\
  &\geq\frac12\eps_{f(k+1)}-\frac{2}{15}\eps_{f(k+1)}
  >\frac14\eps_{f(k+1)}. \qedhere
\end{align*}
\end{proof}

\subsection*{Claim 3}
If $n\notin S$, $k\in[g(n),g(n+1))$ and $\dist(T_k,Q^*_{f(k)}+x)<\frac12\eps_{f(k+1)}$,
then $\dist(\tau,Q^*_{f(k)}+x)\geq\frac14\eps_{f(k+1)}$.
\begin{proof}
This is the only nontrivial case. Geometrically, an approximation $T_k$ of $\tau$ is
close to some point of $Q^*_{f(k)}+x$, so that the next term $\tau_k$ in the sequence
is chosen large enough to move the new approximation $T_{k+1}$ of $\tau$ away from that point.
We however have to make sure that $T_{k+1}$ does not get close to another point of
$Q^*_{f(k)}+x$. That is why condition~\eqref{b2} is required.

Write $\tau=T_k+\tau_k+t_{k+1}$. We know that $\norm{t_{k+1}}\leq\frac{2}{15}\eps_{f(k+1)}$.
Since $\tau_k\neq0$, we also have $\eps_{f(k+1)}\leq\norm{\tau_k}\leq2\eps_{f(k+1)}$.

By assumption, there is $q_0\in Q^*_{f(k)}$ such that $d(q_0+x,T_k)<\frac12\eps_{f(k+1)}$.
First estimate $d(q_0+x,\tau)$ from above:
\begin{align*}
  d(q_0+x,\tau)&\leq d(q_0+x,T_k)+\norm{\tau_k}+\norm{t_{k+1}}\\
  &\leq\left(\frac12+2+\frac{2}{15}\right)\eps_{f(k+1)}<\frac83\eps_{f(k+1)}.
\end{align*}
Use this inequality to estimate $d(p+x,\tau)$ from below for any $p\in Q^*_{f(k)}$, $p\neq q_0$.
Here condition~\eqref{b2} gets in play: it yields
$d(p+x,q_0+x)=d(p,q_0)\geq 3\eps_{f(k+1)}$ and thus
\begin{align*}
  d(p+x,\tau)&\geq d(p+x,q_0+x)-d(q_0+x,\tau)
  \\
  &\geq 3\eps_{f(k+1)}-\frac83\eps_{f(k+1)}>\frac14\eps_{f(k+1)}.
\end{align*}
It remains to show that $d(q_0+x,\tau)\geq\frac14\eps_{f(k+1)}$:
\begin{align*}
  d(q_0+x,\tau)&\geq \norm{\tau_k}-d(q_0+x,T_k)-\norm{t_{k+1}}\\
  &\geq\eps_{f(k+1)}-\frac12\eps_{f(k+1)}-\frac{2}{15}\eps_{f(k+1)}>\frac14\eps_{f(k+1)}.
  \qedhere
\end{align*}
\end{proof}

\smallskip
We can now finish the proof of Lemma~\ref{polish5}:
Claims 1--3 yield \eqref{s2}. Therefore
$\tau\in H(f,x,\frac14)\setminus H(f{\circ} g,0,1)$: the required contradiction.
\end{proof}
\subsection*{Proof of Theorem~\ref{polish1}}

Let $X\subs\grG$ be $\MM$-additive. We will show that $\uhm^h(X)=0$ for
every gauge $h$.

Begin by recursively building $f\in\UPset$ subject to conditions
\eqref{b1}, \eqref{b2} and \eqref{b3} and
\begin{equation}\label{b4}
    \forall k\in\omega\quad h\bigl(2\eps_{f(k+1)}\bigr)<\frac{2^{-k}}{\abs{Q^*_{f(k)}}^2}.
\end{equation}
Consider the set $H(f,0,\frac14)$. It is meager by Lemma~\ref{polish3}. Therefore
$H(f,0,\frac14)+X$ is meager as well.
By virtue of Lemma~\ref{polish4} there thus exists $g\in\UPset$ and $y\in\grG$
such that $H(f,0,\frac14)+x\subs H(f{\circ} g,y,1)$ for all $x\in X$.
Right invariance of the metric gives $H(f,0,\frac14)+x=H(f,x,\frac14)$.
It follows that $H(f,x,\frac14)\subs H(f{\circ} g,y,1)$ for all $x\in X$.
Lemma~\ref{polish5} thus yields
\begin{equation}\label{m1}
  \forall x\in X\, \fmany n\in\omega\, \exists k\in[g(n),g(n+1))
  \ \dist(Q^*_{f(k)}+x,Q^*_{f(k)}+y)\leq\eps_{f(k+1)}.
\end{equation}
The latter inequality means that there are $p,q\in Q^*_{f(k)}$ such that
$d(p+x,q+y)\leq\eps_{f(k+1)}$ which is, via left invariance of the metric, equivalent to
$$
  x\in B(-p+q+y,\eps_{f(k+1)}).
$$
Therefore~\eqref{m1} can be phrased as follows:
\begin{equation}\label{m2}
  \begin{split}
     \forall x\in X\ \fmany n\in\Nset\ \exists k\in[g(n),g(n+1))&\ \exists p,q\in Q^*_{f(k)}\\
     & x\in B(-p+q+y,\eps_{f(k+1)}).
  \end{split}
\end{equation}
Set
\begin{alignat*}{2}
  &\mc B_k &&=\{B(-p+q+y,\eps_{f(k+1)}):p,q\in Q^*_{f(k)}\},\qquad k\in\omega,\\
  &\mc G_n &&=\bigcup\{\mc B_k:g(n)\leq k<g(n+1)\},\qquad n\in\omega,\\
  &\mc G   &&=\bigcup_{k\in\omega}\mc B_k=\bigcup_{n\in\omega}\mc G_n.
\end{alignat*}
With this setting, \eqref{m2} reads
$$
  \forall x\in X\ \fmany n\in\omega\ \exists G\in\mc G_n\quad x\in G.
$$
For all $k\in\omega$ we have $\abs{\mc B_k}\leq\abs{Q^*_{f(k)}}^2$.
In particular, the families $\mc G_n$ are finite, witnessing that
$\mc G$ is a $\gamma$-groupable cover of $X$.

By Lemma~\ref{gammagr}, in order to prove that $\uhm^h(X)=0$
it is enough to show that $\sum_{G\in\mc G}h(\diam G)$ is finite.
Clearly
$\diam B(-p+q+y,\eps_{f(k+1)}\leq2\eps_{f(k+1)}$.
Therefore \eqref{b4} yields
\begin{align*}
  \sum_{G\in\mc G}h(\diam G)&=
  \sum_{k\in\omega}\sum_{G\in\mc B_k}h(\diam G)\\
  &=
  \sum_{k\in\omega}\sum_{p,q\in Q^*_{f(k)}}h(\diam B(-p+q+y,\eps_{f(k+1)}))\\
  &\leq
  \sum_{k\in\omega}\abs{Q^*_{f(k)}}^2\, h(2\eps_{f(k+1)})
  \leq
  \sum_{k\in\omega}2^{-k}<\infty.
\end{align*}
The proof of Theorem~\ref{polish1} is complete.
\qed

\section{Some consequences}\label{sec:conseq}

\subsection*{Continuous images and cartesian products}

Meager-additive sets are preserved by continuous mappings of Polish groups
and by cartesian products as follows:
\begin{thm}\label{mapping22}
Let $\grG_1$ be a \emph{tsi} Polish group and
$\grG_2$ a locally compact Polish group.
Let $f:\grG_1\to\grG_2$ a continuous mapping.
If $X\subs\grG_1$ is $\MM$-additive, then so is $f(X)$.
\end{thm}
\begin{proof}
Since $X$ is $\MM$-additive it is $\ssmz$ by Theorem~\ref{polish1}.
By Proposition~\ref{mapping1}, $f(X)$ is $\ssmz$ as well.
By Theorem~\ref{thm:action3}, $f(X)$ is $\MM$-additive.
\end{proof}

\begin{thm}\label{cart}
Let $\grG_1,\grG_2$ be \emph{tsi} locally compact groups.
Let $X_1\subs\grG_1$, $X_2\subs\grG_2$.
\begin{enum}
\item If $X_1$ is $\smz$ and $X_2$ is $\MM$-additive, then $X_1\times X_2$ is $\smz$.
\item If $X_1$ and $X_2$ are $\MM$-additive, then so is $X_1\times X_2$.
\end{enum}
\end{thm}
\begin{proof}
(i)
By Theorem \ref{main3a}, $X_2$ is $\ssmz$. By
\cite[Theorem 3.14]{Zin-prism}, a product of a $\smz$ and $\ssmz$ sets is $\smz$.

(ii)
By Theorem \ref{main3a}, both $X_1$ and $X_2$ are $\ssmz$. By
\cite[Theorem 3.14]{Zin-prism}, a product of two $\ssmz$ sets is $\ssmz$.
Now apply Theorem \ref{main3a} to conclude that $X_1\times X_2$ is
$\MM$-additive.
\end{proof}

\begin{coro}
Let $\grG_1,\grG_2$ be \emph{tsi} locally compact groups.
Let $X\subs\grG_1\times\grG_2$.
The following are equivalent.
\begin{enum}
\item $X$ is $\MM$-additive,
\item $\proj_1X$ and $\proj_2X$ are $\MM$-additive,
\item $\proj_1X\times\proj_2X$ is $\MM$-additive.
\end{enum}
\end{coro}
\begin{proof}
Write $X_1=\proj_1 X$, $X_2=\proj_2 X$. (iii)\Implies(i) is trivial.

(i)\Implies(ii) Suppose $X$ is $\MM$ additive.
Then by Theorem~\ref{mapping22} both $X_1$ and $X_2$ are $\MM$ additive.

(ii)\Implies(iii) Suppose $X_1$ and $X_2$ are $\MM$ additive.
Then by Theorem~\ref{cart} $X_1\times X_2$ is $\MM$ additive.
\end{proof}

\subsection*{Binary operations}

Theorems~\ref{thm:action3} and~\ref{main3a} reveal a surprising fact
about $\MM$\nobreakdash-additive sets in \emph{tsi} locally compact groups:

\begin{coro}
Let $\grG$ be a \emph{tsi} locally compact group.
Let $\star$ be a continuous binary operation on $\grG$ such that
$N\star x$  is nowhere dense for each nowhere dense set $N$ and $x\in\grG$.
If $X\subs\grG$ is $\MM$-additive, then $M\star X$ is meager for
each meager set $M$.

In particular, this holds for any action $\phi:\grG\curvearrowright\grG$.
\end{coro}

\subsection*{Universally meager sets}
Recall that a separable metric space $E$ is termed \emph{universally meager}
(Zakrzewski \cite{MR1814112,MR2427418})
if for any perfect Polish spaces $Y,X$ such that $E\subs X$ and every
continuous one--to--one mapping $f:Y\to X$ the set $f^{-1}(E)$ is meager
in $Y$. By~\cite[Proposition 6.11]{Zin-prism} every \ssmz{} set is universally meager.
Thus Theorem~\ref{main3a} yields

\begin{prop}\label{umg}
Every $\MM$-additive set in a \emph{tsi} Polish group is universally meager.
\end{prop}

\subsection*{Transitively meager sets}
Nowik, Scheepers and Weiss~\cite{MR1610427} studied the notion of
transitively meager sets on the line. By the definition,
a set $S\subs\grG$ in a Polish group is \emph{transitively meager}
if
$$
  \forall P\subs\grG\text{ perfect }\exists F\sups S\text{ \si compact }
  \forall x\in\grG\  P\cap xF\text{ is meager in $P$}.
$$
\begin{prop}\label{trans}
Every $\MM$-additive set in a \emph{tsi} Polish group is transitively meager.
\end{prop}
\begin{proof}
This is an easy consequence of a technical \cite[Lemma 6.10]{Zin-prism}:
\begin{lem}[{\cite[Lemma 6.10]{Zin-prism}}]
Let $X,Y,Z$ be perfect Polish spaces and $\phi:Y\to X$ a continuous one--to--one mapping.
Let $\mc F$ be an equicontinuous family of uniformly continuous mappings of $Z$ into $X$.
If $E\subs Z$ is \ssmz{}, then there is a \si compact set
$F\sups E$ such that $\phi^{-1}f(F)$ is meager in $Y$ for all $f\in\mc F$.
\end{lem}
Let $S$ be an $\MM$-additive set in $\grG$. By Theorem~\ref{polish1}
it is \ssmz. Now apply the lemma with
$Y=P$, $X=Z=\grG$, $\phi:P\to\grG$ the identity inclusion,
and $\mc F$ the family of all left translates.
\end{proof}

\subsection*{$\gamma$-sets}
Recall the notion of $\gamma$-set, as introduced by
Gerlits and Nagy~\cite{MR667661}.
A family $\mc U$ of open sets in a separable metric space $X$ is called
an \emph{$\omega$-cover} of $X$ if every finite subset of $X$ is
contained in some $U\in\mc U$.
A metric space $X$ is a \emph{$\gamma$-set} if every $\omega$-cover of $X$
contains a $\gamma$-cover.
Generalizing a result of
Nowik and Weiss~\cite[Proposition 3.7]{MR1905154}, ~\cite[Proposition 6.6]{Zin-prism}
proved that every $\gamma$-set is \ssmz{}.
This, together with Theorem~\ref{thm:action3}, yields

\begin{prop}\label{gamma}
Every $\gamma$-set in a locally compact Polish group is sharply $\MM$-additive.
\end{prop}

\section{Uniformity number of meager-additive sets}
\label{sec:uniformity}

Let $\grG$ be a Polish group. We have considered the following \si ideals of small sets:
\begin{itemyze}
\item $\MM^*(\grG)$, the ideal of $\MM$-additive sets in $\grG$,
\item $\MM^\sharp(\grG)$, the ideal of sharply $\MM$-additive sets in $\grG$,
\item $\ssmz(\grG)$, the ideal of $\ssmz$ sets in $\grG$.
This notion may depend, unlike the other two, on the metric.
\end{itemyze}
We will calculate uniformity numbers of $\MM^*(\grG)$ and $\ssmz(\grG)$
for \emph{tsi} Polish groups. Recall that
if $\JJ$ is an ideal of subsets of a set $X$, then
\begin{itemyze}
\item $\non\JJ=\min\{\abs A:A\notin\JJ\}$
(\emph{uniformity} of $\JJ$)
\item $\add\JJ=\min\{\abs{\mc A}:\mc A\subs\JJ,\bigcup\mc A\notin\JJ\}$
(\emph{additivity} of $\JJ$)
\item $\cov\JJ=\min\{\abs{\mc A}:\mc A\subs\JJ,\bigcup\mc A=\grG\}$
(\emph{covering} of $\JJ$)
\end{itemyze}

The invariants $\non\MM(\Cset)$, $\add\MM(\Cset)$ and $\cov\MM(\Cset)$
are commonly denoted just $\non\MM$, $\add\MM$ and $\cov\MM$.

The cardinal invariant $\nonM^*(\Cset)$ is termed \emph{transitive additivity of $\MM$}
and denoted $\addM^*$ in~\cite{MR1350295,MR776210}, where a combinatorial
characterization of $\addM^*$ is also provided (\cite[Theorem 2.7.14]{MR1350295}).

\begin{thm}\label{invariants}
Let $\grG$ be a \emph{tsi} Polish group.
\begin{enum}
\item If $\grG$ is locally compact, then
$\non\ssmz(\grG)=\nonM^*(\grG)=\add^*\MM$.
\item If $\grG$ is not locally compact, then
$\non\ssmz(\grG)=\nonM^*(\grG)=\addM$.
\end{enum}
\end{thm}
\begin{proof}
(i)
Let $K_n\upto\grG$ be a sequence of compact sets. Let $X\subs\grG$ be
a not $\ssmz$ set such that $\abs X=\non\ssmz(\grG)$.
There is $n$ such that $X\cap K_n$ is not $\ssmz$, therefore
$\abs{X\cap K_n}=\non\ssmz(\grG)$. Since $K_n$ is compact, it is \emph{dyadic}:
there is a continuous mapping $f:\Cset\to K_n$ onto $K_n$.
For each $x\in X\cap K_n$ pick $\widehat x\in f^{-1}(x)$ and set
$\widehat X=\{\widehat x:x\in X\cap K_n\}$.
Then $f(\widehat X)=X\cap K_n$, hence $\widehat X$ is by Proposition~\ref{mapping1}
not $\ssmz$, and clearly $\abs{\widehat X}=\abs{X\cap K_n}=\non\ssmz(\grG)$.
It follows that
$\non\ssmz(\Cset)\leq\non\ssmz(\grG)$.

On the other hand, by the Perfect set theorem, $\grG$ contains a copy of $\Cset$.
It follows that
$\non\ssmz(\Cset)\geq\non\ssmz(\grG)$.

Combine the two inequalities to get $\non\ssmz(\Cset)=\non\ssmz(\grG)$.
By Theorem~\ref{main3a} we have $\ssmz(\Cset)=\MM^*(\grG)$ and likewise for $\Cset$.
Thus $\non\ssmz(\Cset)=\non\MM^*(\Cset)=\add^*\MM$ and
$\non\ssmz(\grG)=\non\MM^*(\grG)$ and the result follows.

(ii)
We will consider $\ssmz$ sets in $\Pset$. The metric on $\Pset$ is the usual
least difference metric: for distinct $f,g\in\Pset$ we let $d(f,g)=2^{-n}$,
where $n$ is the first number for which $f(n)\neq g(n)$.

Recall that a set $X\subs\Pset$ is \emph{bounded} if
$\exists g\in\Pset\ \forall f\in X\ \fmany n\in\Nset\ f(n)\leq g(n)$.
Let $\bbb=\min\{\abs X:X\subs\Pset\text{ is not bounded}\}$
denote the \emph{bounding number}.

Let $X\subs\Pset$ be an unbounded set such that $\abs X=\bbb$.
Since $X$ is not bounded, it is not contained in a \si compact set
and in particular (by Lemma~\ref{totbd}) it is not $\ssmz$.
It follows that $\non\ssmz(\Pset)\leq\bbb$.

Fremlin and Miller~\cite{MR954892} constructed a set $X\subs\Pset$ that is not $\smz$
and $\abs X=\cov\MM$. This set is clearly not $\ssmz$ and thus
$\non\ssmz(\Pset)\leq\cov\MM$.

It follows that $\non\ssmz(\Pset)\leq\min(\cov\MM,\bbb)$. Miller~\cite{MR613787}
proved that $\min(\cov\MM,\bbb)=\addM$ (see also \cite[2.2.9]{MR1350295}).
Overall we have
\begin{equation}\label{card2}
  \non\ssmz(\Pset)\leq\addM.
\end{equation}
Our group $\grG$ is not locally compact and is equipped with an invariant metric.
By~\cite[Lemma 5.4]{MR3453581}, such a group contains a uniform copy of $\Pset$
and thus $\non\ssmz(\grG)\leq\non\ssmz(\Pset)$. Combine with
\eqref{card2} to get
\begin{equation}\label{card3}
  \non\ssmz(\grG)\leq\addM.
\end{equation}

It is well-known that $\addM(X)=\addM$ for every uncountable Polish space.
Thus if $X$ is a set in $\grG$ such that $\abs X<\addM$, then
for each meager set $M\subs\grG$ the set $MX$ is a union of less than
$\addM$ many meager sets and is thus meager. It follows that
\begin{equation}\label{card4}
  \non\MM^*(\grG)\geq\addM.
\end{equation}
Finally, by Theorem~\ref{polish1}, $\MM^*(\grG)\subs\ssmz(\grG)$ and hence
\begin{equation}\label{card5}
  \non\MM^*(\grG)\leq\non\ssmz(\grG).
\end{equation}
Now combine \eqref{card4},\eqref{card5} and \eqref{card3}:
$$
 \addM\leq\non\MM^*(\grG)\leq\non\ssmz(\grG)\leq\addM
$$
and the result follows.
\end{proof}

\begin{rem}\label{nottop}
Theorem~\ref{invariants} provides a simple argument proving that $\ssmz$
is consistently not a topological property. (Remind that under Borel Conjecture
$\ssmz$ is a topological property for a trivial reason.)

The Baer-Specker group $\Zset^\Nset$ is a \emph{tsi} Polish group.
On the other hand, it is homeomorphic to the set
of irrational numbers, so regard it as a subset of $\Rset$.
By Theorem~\ref{invariants}(ii) there is a set $X\subs\Zset^\Nset$ such that
$\abs X=\addM$
and that is not $\ssmz$ in the invariant metric.
On the other hand, if $\addM<\add^*\MM$, then $X$ is by Theorem~\ref{invariants}(i)
$\ssmz$ in the metric of the real line. Since $\addM<\add^*\MM$
is relatively consistent with ZFC (as proved by Pawlikowski~\cite{MR776210}),
$X$ is consistently a set that is $\ssmz$ is one metric on $\Zset^\Nset$
and not $\ssmz$ in another homeomorphic metric.
\end{rem}

\section{Questions}\label{sec:questions}

We conclude with a few open problems.

\subsection*{$\MM$-additive vs.~$\ssmz$ under the Continuum Hypothesis}

As to $\smz$ sets in groups, let us recall that
Galvin-Mycielski-Solovay Theorem has been extended to all
locally compact Polish groups by Kysiak \cite{kysiak} and Fremlin \cite{fremlin5}.
Further extension of the theorem to Polish groups has been proven to be impossible:
Hru\v s\'ak, Wohofsky and Zindulka~\cite{MR3453581}
and Hru\v s\'ak and Zapletal~\cite{MR3707641} found that under the
Continuum Hypothesis the Galvin--Mycielski--Solovay Theorem fails for
any \emph{tsi} Polish groups that are not locally compact.

We ask if an analogous situation occurs for $\ssmz$ sets.
We know from Theorem~\ref{polish1} that for a \emph{tsi} Polish group
we have $\MM^*(\grG)\subs\ssmz(\grG)$. Does the converse inclusion hold?

\begin{question}
Assume the Continuum Hypothesis.
Is there a \emph{tsi}, not locally compact Polish group $\grG$ such that
$\MM^*(\grG)\neq\ssmz(\grG)$?
\end{question}

We conjecture the answer to be negative. More specifically, we may ask about the
Baer-Specker group:

\begin{question}
Assume the Continuum Hypothesis.
Is there a $\ssmz$ set in $\Zset^\Nset$ that is not $\MM$-additive?
\end{question}

\subsection*{\emph{tsi} condition}

The proof of Theorem~\ref{polish1} relies on the two-sided invariance
of the metric. Can we do better?
\begin{question}
Do Theorems~\ref{polish1} and~\ref{main3a} remain valid for all
Polish groups? In other words, can the assumption that the group is
\emph{tsi} be dropped or relaxed to the existence of a
complete left-invariant metric?
\end{question}

\subsection*{$\EE$-additive sets}
Let $\grG$ be locally compact Polish group. Denote by $\EE$ the \si ideal
consisting of sets that are contained in an $F_\sigma$-set of Haar measure zero.
A set $X\subs\grG$ is \emph{$\EE$-additive} if $EX\in\EE$ for every $E\in\EE$.
The notion of \emph{sharply $\EE$-additive} sets can be defined, too,
in the obvious manner.

In~\cite[Theorem 5.2]{Zin-prism} it is shown that a set $X\subs\Cset$
is $\MM$-additive if and only if it is $\EE$-additive if and only if
it is sharply $\EE$-additive. The proofs heavily depend on
the combinatorics of sets in $\EE$.
We have no idea if this theorem or at least some inclusions
extend to other locally compact groups.

\begin{question}\label{qf}
\begin{enum}
\item In what locally compact Polish groups is every $\MM$-additive set
$\EE$-additive (or sharply $\EE$-additive)?
\item In what locally compact Polish groups is every $\EE$-additive
(or sharply $\EE$-additive) set $\MM$-additive?
\item In what locally compact Polish groups is every $\EE$-additive
sharply $\EE$-additive?
\end{enum}
\end{question}

The two new classes extend the family of ideals that we are interested in to five:
$\ssmz$, $\MM^*$, $\MM^\sharp$, $\EE^*$ and $\EE^\sharp$,
leading to a number of questions about possible inclusions.
For \emph{tsi} locally compact groups, the twelve questions reduce,
due to Theorem~\ref{main3a}, to the five listed in Question \ref{qf}.

\section*{Acknowledgments}
I would like to thank Michael Hru\v s\'ak for his kind support and
help while I was
visiting Instituto de matem\'aticas, Unidad Morelia,
Universidad Nacional Auton\'oma de M\'exico;
and Piotr Szewczak and Boaz Tsaban who encouraged me to write the paper.
\bibliographystyle{amsplain}
\providecommand{\bysame}{\leavevmode\hbox to3em{\hrulefill}\thinspace}
\providecommand{\MR}{\relax\ifhmode\unskip\space\fi MR }
\providecommand{\MRhref}[2]{%
  \href{http://www.ams.org/mathscinet-getitem?mr=#1}{#2}
}
\providecommand{\href}[2]{#2}

\end{document}